\documentclass[a4paper, 10pt, conference]{ieeeconf}      

\IEEEoverridecommandlockouts                              
\renewcommand{\thesection}{ \arabic{section}}
\renewcommand{\thesubsection}{ \arabic{section}}
\newtheorem{theorem}{Theorem}[section]

\newtheorem{lemma}[theorem]{Lemma}

\newtheorem{definition}[theorem]{Definition}
\newtheorem{remark}[theorem]{Remark}

\newtheorem{algorithm}{\normalfont\textsc{Algorithm}}
\newtheorem{assumption}{Assumption A.\!\!}

\newtheorem{scheme}{\noindent\normalfont\textsc{Scheme A.\!\!}}

\newcommand{\h}{\mathcal{H}}

\newcommand{\norm}[1]{\|#1\|}

\renewcommand{\int}{\mathrm{int}}

\usepackage{graphics} 
\usepackage{epsfig} 
\usepackage{mathptmx} 
\usepackage{amsmath} 
\usepackage{amssymb}  

\usepackage{graphicx}      
\usepackage{color}         
\usepackage{ifpdf}
\usepackage{url}

\usepackage[top=2.2cm,bottom=2.8cm,right=1.8cm,left=1.8cm]{geometry}

\usepackage{tikz}
\usetikzlibrary{calc}

\title{\LARGE{\textbf{An Inner Convex Approximation Algorithm for BMI Optimization and Applications in Control}}}

\author{ \parbox{4 in}{\centering Quoc Tran Dinh$^{\dagger *}$, Wim Michiels$^{\ddagger}$ and Moritz Diehl$^{\dagger}$
         \thanks{$^{\dagger}$Department of Electrical Engineering (ESAT/SCD) and Optimization in Engineering Center (OPTEC), Katholieke Universiteit Leuven,
		Belgium. {\tt\footnotesize{Email: \{quoc.trandinh, moritz.diehl\}@esat.kuleuven.be}}
		\newline{$^{\ddagger}$Department of Computer Science and Optimization in Engineering Center (OPTEC), KU Leuven, Belgium.
{\tt\footnotesize{Email: wim.michiels@cs.kuleuven.be}}}
		\newline{$^{*}$Department of Mathematics-Mechanics-Informatics, Hanoi University of Science, Hanoi, Vietnam.}}
         }}

\begin{document}

\maketitle
\thispagestyle{empty}
\pagestyle{empty}

\begin{abstract}                          
In this work, we propose a new local optimization method to solve a class of nonconvex semidefinite programming (SDP) problems. The basic idea is to 
approximate the feasible set of the nonconvex SDP problem by inner positive semidefinite convex approximations via a parameterization technique. This leads to
an iterative procedure to search a local optimum of the nonconvex problem. The convergence of the algorithm is analyzed under mild assumptions. 
Applications in static output feedback control are benchmarked and numerical tests are implemented based on the data from the COMPL$_{\mathrm{e}}$ib library.
\end{abstract}


\section{Introduction}
We are interested in the following nonconvex semidefinite programming problem:
\begin{equation}\label{eq:NSDP}
\left\{\begin{array}{cl}
\displaystyle\min_{x\in\mathbb{R}^n} & f(x)\\
\textrm{s.t.} & F_i(x) \preceq 0, ~~ i=1,\dots, m, \\
&x \in \Omega,
\end{array}\right.\tag{$\mathrm{NSDP}$}
\end{equation}
where $f : \mathbb{R}^n\to\mathbb{R}$ is convex, $\Omega$ is a nonempty, closed convex set in $\mathbb{R}^n$ and $F_i
:\mathbb{R}^n\to \mathcal{S}^{p_i}$ ($i=1,\dots, m$) are nonconvex matrix-valued mappings and smooth.
The notation $A\preceq 0$ means that $A$ is a symmetric negative semidefinite matrix.
Optimization problems involving matrix-valued mapping inequality constraints have large number of applications in static output feedback controller
design and topology optimization, see, e.g. \cite{Boyd1994,Kocvara2005,Leibfritz2004,String2009}. 
Especially, optimization problems with bilinear matrix inequality (BMI) constraints have been known to be nonconvex and NP-hard~\cite{Blondel1997}. Many
attempts have been done to solve these problems by employing convex semidefinite programming (in particular, optimization with linear matrix inequality (LMI)
constraints) techniques~\cite{Correa2004,Freund2007,Kocvara2005,Leibfritz2002,Thevenet2006}. The methods developed in those papers are based on
augmented
Lagrangian functions, generalized sequential semidefinite programming and alternating directions. Recently, we proposed a new method based on convex-concave
decomposition of the BMI constraints and linearization technique \cite{TranDinh2011}. The method exploits the convex substructure of the problems.
It was shown that this method can be applied to solve many problems arising in static output feedback control including spectral abscissa, $\h_2$, $\h_{\infty}$
and mixed $\h_2/\h_{\infty}$ synthesis problems.  

In this paper, we follow the same line of the work in \cite{Beck2010,Marks1978,TranDinh2011} to develop a new local optimization method for solving the
nonconvex semidefinite programming problem \eqref{eq:NSDP}. 
The main idea is to approximate the feasible set of the nonconvex problem by a sequence of inner positive semidefinite convex approximation sets.
This method can be considered as a generalization of the ones in \cite{Beck2010,Marks1978,TranDinh2011}.

\vskip0.1cm
\noindent\textit{Contribution. } 
The contribution of this paper can be summarized as follows:
\begin{itemize}
\item[$\mathrm{1.}$] We generalize the inner convex approximation method in \cite{Beck2010,Marks1978} from scalar optimization to nonlinear semidefinite
programming. 
Moreover, the algorithm is modified by using a \textit{regularization technique} to ensure strict descent. 
The advantages of this algorithm are that it is \textit{very simple to implement} by employing available standard semidefinite programming software tools and
\textit{no globalization strategy} such as a line-search procedure is needed.
\item[$\mathrm{2.}$] We prove the convergence of the algorithm to a stationary point under mild conditions.
\item[$\mathrm{3.}$] We provide two particular ways to form an overestimate for bilinear matrix-valued mappings and then show many applications in static
output feedback.
\end{itemize}

\vskip0.1cm
\noindent\textit{Outline. } 
The next section recalls some definitions, notation and properties of matrix operators and defines an inner convex approximation of a BMI
constraint. Section \ref{sec:alg_and_conv} proposes the main algorithm and investigates its convergence properties. Section \ref{sec:app} shows the applications
in static output feedback control and numerical tests. Some concluding remarks are given in the last section.

\section{Inner convex approximations}\label{sec:inner_cv_app}
In this section, after given an overview on concepts and definitions related to matrix operators, we provide a definition of inner positive semidefinite convex
approximation of a nonconvex set.

\subsection{Preliminaries}
Let $\mathcal{S}^p$ be the set of symmetric matrices of size $p\times p$, $\mathcal{S}^p_{+}$, and resp., $\mathcal{S}^p_{++}$ be the set of symmetric positive
semidefinite, resp.,  positive definite matrices. For
given matrices $X$ and $Y$ in $\mathcal{S}^p$, the relation $X\succeq Y$ (resp., $X\preceq Y$) means that $X-Y\in\mathcal{S}^p_{+}$ (resp.,
$Y-X\in\mathcal{S}^p_{+}$) and $X\succ Y$ (resp., $X\prec
Y$)
is $X-Y\in\mathcal{S}^p_{++}$ (resp., $Y-X\in\mathcal{S}^p_{++}$). The quantity $X\circ Y := \textrm{trace}(X^TY)$ is an inner product of two matrices $X$ and
$Y$ defined on $\mathcal{S}^p$, where $\textrm{trace}(Z)$ is the trace of matrix $Z$.
For a given symmetric matrix $X$, $\lambda_{\min}(X)$ denotes the smallest eigenvalue of $X$.

\begin{definition}\label{de:psd_convex}\cite{Shapiro1997}
A matrix-valued mapping $F : \mathbb{R}^n\to\mathcal{S}^p$ is said to be positive semidefinite convex (\textit{psd-convex}) on a convex subset
$C\subseteq\mathbb{R}^n$ if for all $t\in[0,1]$ and
$x,y\in C$, one has
\begin{equation}\label{eq:psd_convex}
F(tx+(1-t)y) \preceq tF(x) + (1-t)F(y).
\end{equation}
If \eqref{eq:psd_convex} holds for $\prec$ instead of $\preceq$ for $t\in (0,1)$ then $F$ is said to be \textit{strictly psd-convex} on $C$.
In the opposite case, $F$ is said to be \textit{psd-nonconvex}.
Alternatively, if we replace $\preceq$ in \eqref{eq:psd_convex} by $\succeq$ then $F$ is said to be psd-concave on $C$.
It is obvious that any convex function $f : \mathbb{R}^n\to\mathbb{R}$ is psd-convex with $p=1$.
\end{definition}
A function $f :\mathbb{R}^n\to\mathbb{R}$ is said to be \textit{strongly convex} with parameter $\rho > 0$ if $f(\cdot) - \frac{1}{2}\rho\norm{\cdot}^2$ is
convex. The notation $\partial{f}$ denotes the subdifferential of a convex function $f$. For a given convex set $C$, $\mathcal{N}_C(x) := \left\{ w
~|~ w^T(x-y) \geq 0, ~y\in C\right\}$ if $x\in C$ and $\mathcal{N}_C(x) := \emptyset$ if $x\notin C$ denotes the normal cone of $C$ at $x$.


The derivative of a matrix-valued mapping $F$ at $x$ is a linear mapping $DF$ from $\mathbb{R}^n$ to $\mathbb{R}^{p\times p}$ which is defined by
\begin{equation*}
DF(x)h := \sum_{i=1}^nh_i\frac{\partial{F}}{\partial x_i}(x), ~\forall h\in \mathbb{R}^n.
\end{equation*}
For a given convex set $X\in\mathbb{R}^n$, the matrix-valued mapping $G$ is said to be differentiable on a subset $X$ if its derivative $DF(x)$ exists at every
$x\in X$. The definitions
of the second order derivatives of matrix-valued mappings can be found, e.g., in \cite{Shapiro1997}.
Let $A : \mathbb{R}^n\to \mathcal{S}^p$ be a linear mapping defined as $Ax := \sum_{i=1}^nx_iA_i$, where $A_i \in \mathcal{S}^p$ for $i=1,\dots, n$.
The adjoint operator of $A$, $A^{*}$,  is defined as $A^{*}Z := (A_1\circ Z, A_2\circ Z, \dots, A_n\circ Z)^T$ for any $Z\in\mathcal{S}^p$.

Finally, for simplicity of discussion, throughout this paper, we assume that all the functions and matrix-valued mappings are \textit{twice differentiable} on
their domain. 

\subsection{Psd-convex overestimate of a matrix operator}
Let us first describe the idea of the inner convex approximation for the scalar case. 
Let $f :\mathbb{R}^n\to\mathbb{R}$ be a continuous nonconvex function. 
A convex function $g(\cdot;y)$ depending on a parameter $y$  is called a convex overestimate of $f(\cdot)$ w.r.t. the parameterization $y := \psi(x)$  if
$g(x, \psi(x)) = f(x)$ and $f(z) \leq g(z; y)$ for all $y, z$.
Let us consider two examples.

\vskip0.1cm
\noindent\textit{Example 1. }
Let $f$  be a continuously differentiable function and its gradient $\nabla{f}$ is Lipschitz continuous with a Lipschitz constant
$L_f>0$, i.e. $\norm{\nabla{f}(y)-\nabla{f}(x)}\leq L\norm{y-x}$ for all $x, y$. Then, it is well-known that $|f(z) - f(x) - \nabla{f}(x)^T(z-x)| \leq
\frac{L_f}{2}\norm{z-x}^2$. Therefore, for any $x, z$ we have $f(z) \leq g(z;x)$ with $g(z;x) := f(x) + \nabla{f}(x)^T(z-x) + \frac{L_f}{2}\norm{z-x}^2$.
Moreover, $f(x) = g(x;x)$ for any $x$.
We conclude that $g(\cdot;x)$ is a convex overestimate of $f$ w.r.t the parameterization $y = \psi(x) = x$.
Now, if we fix $x=\bar{x}$ and find a point $v$ such that $g(v;\bar{x}) \leq 0$ then $f(v) \leq 0$. Consequently if the set $\{x ~|~ f(x) < 0\}$ is nonempty,
we can find a point $v$ such that $g(v;\bar{x}) \leq 0$. The convex set $\mathcal{C}(x) := \{z ~|~ g(z;x)\leq 0\}$ is called an inner convex approximation
of $\{z~|~f(z)\leq 0\}$.

\vskip0.1cm
\noindent\textit{Example 2. }\cite{Beck2010} 
We consider the function $f(x) = x_1x_2$ in $\mathbb{R}^2$. The function $g(x, y) = \frac{y}{2}x_1^2 +
\frac{1}{2y}x_2^2$ is a convex overestimate of $f$ w.r.t. the parameterization $y = \psi(x) = x_1/x_2$ provided that $y > 0$.
This example shows that the mapping $\psi$ is not always identity.

Let us generalize the convex overestimate concept to matrix-valued mappings.

\begin{definition}\label{def:over_relaxation}
Let us consider a psd-nonconvex matrix mapping $F : \mathcal{X}\subseteq\mathbb{R}^n\to\mathcal{S}^p$. A psd-convex matrix mapping
$G(\cdot; y)$ is said to be a psd-convex overestimate of $F$ w.r.t. the parameterization $y := \psi(x)$ if $G(x; \psi(x)) = F(x)$ and $F(z) \preceq
G(z;y)$ for all $x, y$ and $z$ in $\mathcal{X}$.
\end{definition}
Let us provide two important examples that satisfy Definition \ref{def:over_relaxation}.

\noindent\textit{Example 3. } 
Let $\mathcal{B}_Q(X, Y) = X^TQ^{-1}Y + Y^TQ^{-1}X$ be a bilinear form with $Q = Q_1 + Q_2$, $Q_1\succ 0$ and $Q_2\succ 0$ arbitrarily, where $X$ and $Y$ are
two $n\times p$ matrices.
We consider the parametric quadratic form:
\begin{align}\label{eq:G_Q}
\mathcal{Q}_{Q}(X, Y;\bar{X},\bar{Y}) & :={\!\!} (X \!-\!\bar{X})^T\!\!Q_1^{-1}(X \!-\!\bar{X})\!+\! (Y \!-\!\bar{Y})^T\!\!Q_2^{-1}(Y \!-\!
\bar{Y})\nonumber\\
& +\! \bar{X}^T\!Q^{-1}Y \!+\! \bar{Y}^TQ^{-1}X + X^TQ^{-1}\bar{Y} \\
& +\! Y^TQ^{-1}\bar{X}  - \bar{X}^TQ^{-1}\bar{Y} \!-\! \bar{Y}^TQ^{-1}\bar{X}.\nonumber
\end{align}
One can show that $\mathcal{Q}_Q(X, Y; \bar{X}, \bar{Y})$ is a psd-convex overestimate of $\mathcal{B}_Q(X,Y)$ w.r.t. the parameterization
$\psi(\bar{X},\bar{Y})=(\bar{X},\bar{Y})$.

Indeed, it is obvious that $\mathcal{Q}_Q(\bar{X}, \bar{Y}; \bar{X}, \bar{Y}) = \mathcal{B}_Q(\bar{X}, \bar{Y})$. We only prove the second condition in
Definition \ref{def:over_relaxation}.
We consider the expression $\mathcal{D}_Q := \bar{X}^TQ^{-1}Y + \bar{Y}^TQ^{-1}X + X^TQ^{-1}\bar{Y} + Y^TQ^{-1}\bar{X} - \bar{X}^TQ^{-1}\bar{Y} -
\bar{Y}^TQ\bar{X} - X^TQ^{-1}Y - Y^TQ^{-1}X$. By rearranging this expression, we can easily show that $\mathcal{D}_Q = -(X-\bar{X})^TQ^{-1}(Y-\bar{Y}) -
(Y-\bar{Y})^TQ^{-1}(X-\bar{X})$.  
Now, since $Q = Q_1 + Q_2$, by \cite{Bernstein2005}, we can write:
\begin{eqnarray}\label{eq:lm21_proof1}
-\mathcal{D}_Q &&{\!\!\!\!\!\!\!\!\!\!\!\!}= (X-\bar{X})^T(Q_1+Q_2)^{-1}(Y-\bar{Y}) \nonumber\\
&&{\!\!\!\!\!\!\!\!\!\!\!\!} + (Y-\bar{Y})^T(Q_1+Q_2)^{-1}(X-\bar{X}) \\
&&{\!\!\!\!\!\!\!\!\!\!\!\!} \preceq \! (X \!-\! \bar{X})^TQ_1^{-1}\!(X \!-\!\bar{X}) \!+\! (Y \!-\! \bar{Y})^T\!Q_2^{-1}\!(Y \!-\! \bar{Y}).\nonumber
\end{eqnarray}
Note that $\mathcal{D}_Q = \mathcal{Q}_Q - \mathcal{B}_Q - (X \!-\! \bar{X})^TQ_1^{-1}\!(X \!-\!\bar{X}) \!+\! (Y \!-\! \bar{Y})^T\!Q_2^{-1}\!(Y \!-\!
\bar{Y})$. Therefore, we have $\mathcal{Q}_Q(X, Y;\bar{X}, \bar{Y}) \succeq \mathcal{B}_Q(X, Y)$ for all $X, Y$ and $\bar{X},\bar{Y}$.


\noindent\textit{Example 4. } Let us consider a psd-noncovex matrix-valued mapping $\mathcal{G}(x) := \mathcal{G}_{\mathrm{cvx1}}(x) -
\mathcal{G}_{\mathrm{cvx2}}(x)$, where $\mathcal{G}_{\mathrm{cvx1}}$ and $\mathcal{G}_{\mathrm{cvx2}}$ are two psd-convex matrix-valued mappings
\cite{TranDinh2011}. Now, let $\mathcal{G}_{\mathrm{cvx2}}$ be differentiable and $\mathcal{L}_2(x;\bar{x}) := \mathcal{G}_{\mathrm{cvx2}}(\bar{x}) +
D\mathcal{G}_{\mathrm{cvx2}}(\bar{x})(x-\bar{x})$ be the linearization of $\mathcal{G}_{\mathrm{cvx2}}$ at $\bar{x}$. We define $\mathcal{H}(x;\bar{x}) :=
\mathcal{G}_{\mathrm{cvx1}}(x) - \mathcal{L}_2(x;\bar{x})$. It is not difficult to show that $\mathcal{H}(\cdot;\cdot)$ is a psd-convex overestimate of
$\mathcal{G}(\cdot)$ w.r.t. the parametrization $\psi(\bar{x}) = \bar{x}$.

\begin{remark}\label{re:nonunique_of_BMI_app}
\textit{Example 3} shows that the ``Lipschitz coefficient'' of the approximating function \eqref{eq:G_Q} is $(Q_1,Q_2)$.
Moreover, as indicated by \textit{Examples} 3 and 4, the psd-convex overestimate of a bilinear form is not unique. In practice, it is important to find 
appropriate psd-convex overestimates for bilinear forms to make the algorithm perform efficiently.
Note that the psd-convex overestimate $\mathcal{Q}_{Q}$ of $\mathcal{B}_{Q}$ in \textit{Example 3} may be less conservative than the convex-concave
decomposition in \cite{TranDinh2011} since all the terms in $\mathcal{Q}_{Q}$ are related to $X-\bar{X}$ and $Y-\bar{Y}$ rather than $X$ and $Y$.
\end{remark}


\section{The algorithm and its convergence}\label{sec:alg_and_conv}
Let us recall the nonconvex semidefinite programming problem \eqref{eq:NSDP}. We denote by 
\begin{eqnarray}\label{eq:Feasible_set}
\mathcal{F} := \left\{ x\in\Omega ~|~ F_i(x) \preceq 0, ~i=1,\dots, m\right\},
\end{eqnarray}
the feasible set of \eqref{eq:NSDP} and 
\begin{equation}
\mathcal{F}^0 \!\!:=\! \textrm{ri}(\Omega) \!\cap\! \left\{x \in \mathbb{R}^n \!~|~\! F_i(x) \!\prec\! 0, ~i=1,\dots, m \right\},
\end{equation}
the relative interior of $\mathcal{F}$, where $\textrm{ri}(\Omega)$ is the relative interior of $\Omega$. 
First, we need the following fundamental assumption.

\begin{assumption}\label{as:A1}
The set of interior points $\mathcal{F}^0$ of $\mathcal{F}$ is nonempty.
\end{assumption}
Then, we can write the generalized KKT system of \eqref{eq:NSDP} as follows:
\begin{equation}\label{eq:KKT_sys}
\begin{cases}
0 \in \partial{f}(x) + \sum_{i=1}^mDF_i(x)^{*}W_i + \mathcal{N}_{\Omega}(x),\\
0 \!\succeq\! F_i(x), ~ W_i \!\succeq\! 0, ~ F_i(x) \!\circ\! W_i \!=\! 0, ~i=1,\dots, m.
\end{cases}
\end{equation}
Any point $(x^{*}, W^{*})$ with $W^{*} := (W^{*}_1,\dots, W^{*}_m)$ is called a \textit{KKT point} of \eqref{eq:NSDP}, where $x^{*}$ is called a
\textit{stationary point} and $W^{*}$ is called the corresponding Lagrange multiplier.

\subsection{Convex semidefinite programming subproblem}
The main step of the algorithm is to solve a convex semidefinite programming problem formed at the iteration
$\bar{x}^k\in\Omega$ by using inner psd-convex approximations. This problem is defined as follows:
\begin{equation}\label{eq:convx_subprob}
\left\{\begin{array}{cl}
\displaystyle\min_{x} & f(x) + \frac{1}{2}(x-\bar{x}^k)^TQ_{k}(x-\bar{x}^k)\\
\textrm{s.t.}~ &G_i(x; \bar{y}^k_i) \preceq 0, ~i=1,\dots, m\\
&x \in \Omega. 
\end{array}\right.\tag{$\mathrm{CSDP}(\bar{x}^k)$}
\end{equation}
Here, $Q_k\in\mathcal{S}^n_{+}$ is given and the second term in the objective function is referred to as a regularization term;
$\bar{y}_i^k :=
\psi_i(\bar{x}^k)$ is the parameterization of the convex overestimate $G_i$ of $F_i$.

Let us define by $\mathcal{S}(\bar{x}^k, Q_k)$ the solution mapping of \ref{eq:convx_subprob} depending on the parameters $(\bar{x}^k, Q_k)$. Note that the
problem \ref{eq:convx_subprob} is convex, $\mathcal{S}(\bar{x}^k; Q_k)$ is multivalued and convex.
The feasible set of \ref{eq:convx_subprob} is written as:
\begin{equation}\label{eq:feasible_k}
\mathcal{F}(\bar{x}^k) := \left\{ x\in\Omega ~|~ G_i(x; \psi_i(\bar{x}^k)) \preceq 0, ~i=1,\dots, m\right\}.  
\end{equation}

\subsection{The algorithm}
The algorithm for solving \eqref{eq:NSDP} starts from an initial point $\bar{x}^0\in \mathcal{F}^0$ and generates a sequence $\{\bar{x}^k\}_{k\geq 0}$ by
solving a sequence of convex semidefinite programming subproblems \ref{eq:convx_subprob} approximated at $\bar{x}^k$. More precisely, it is presented in
detail as follows.

\begin{algorithm}[\textit{Inner Convex Approximation}]\label{alg:A1}
\noindent\textbf{Initialization.} Determine an initial point $\bar{x}^0\in\mathcal{F}^0$. Compute $\bar{y}^0_i := \psi_i(\bar{x}^0)$ for $i=1,\dots, m$. Choose
a regularization matrix $Q_0\in\mathcal{S}^n_{+}$. Set $k:=0$.

\noindent\textbf{Iteration $k$ ($k=0,1,\dots$)} Perform the following steps: 
\begin{itemize}
\item[]\textit{Step 1. } For given $\bar{x}^k$, if a given criterion is satisfied then terminate.

\item[]\textit{Step 2. } Solve the convex semidefinite program \ref{eq:convx_subprob} to obtain a solution $\bar{x}^{k+1}$ and the corresponding
Lagrange multiplier $\bar{W}^{k+1}$.
\item[]\textit{Step 3. } Update $\bar{y}_i^{k+1} := \psi_i(\bar{x}^{k+1})$, the regularization matrix $Q_{k+1}\in\mathcal{S}^n_{+}$ (if
necessary). 
Increase $k$ by $1$ and go back to Step 1.
\end{itemize}
\noindent\textbf{End.}
\end{algorithm}

The core step of Algorithm \ref{alg:A1} is Step 2 where a general convex semidefinite program needs to be solved. In practice, this can be done by either
implementing a particular method that exploits problem structures or relying on standard semidefinite programming software tools.
Note that the regularization matrix $Q_k$ can be fixed at $Q_k = \rho I$, where $\rho > 0$ is sufficiently small and $I$ is the identity matrix.
Since Algorithm \ref{alg:A1} generates a feasible sequence $\{\bar{x}^k\}_{k\geq 0}$ to the original problem \eqref{eq:NSDP} and this sequence is
strictly descent w.r.t. the objective function $f$, \textit{no globalization strategy} such as line-search or trust-region is needed.

\subsection{Convergence analysis}
We first show some properties of the feasible set $\mathcal{F}(\bar{x})$ defined by \eqref{eq:feasible_k}. 
For notational simplicity, we use the notation $\norm{\cdot}_Q^2 := (\cdot)^TQ(\cdot)$.

\begin{lemma}\label{le:feasible_set}
Let $\{x^k\}_{k\geq 0}$ be a sequence generated by Algorithm \ref{alg:A1}. Then:
\begin{itemize}
\item[]$\mathrm{a)}$ The feasible set $\mathcal{F}(\bar{x}^k)\subseteq \mathcal{F}$ for all $k \geq 0$.
\item[]$\mathrm{b)}$ It is a feasible sequence, i.e. $\{\bar{x}^k\}_{k\geq 0}\subset\mathcal{F}$.
\item[]$\mathrm{c)}$ $\bar{x}^{k+1} \in \mathcal{F}(\bar{x}^k)\cap \mathcal{F}(\bar{x}^{k+1})$.
\item[]$\mathrm{d)}$ For any $k\geq 0$, it holds that:
\begin{equation*}
f(\bar{x}^{k+1}) \leq f(\bar{x}^k) - \frac{1}{2}\norm{\bar{x}^{k+1} - \bar{x}^k}^2_{Q_k} -
\frac{\rho_f}{2}\norm{\bar{x}^{k+1}-\bar{x}^k}^2, 
\end{equation*}
where $\rho_f\geq 0$ is the strong convexity parameter of $f$.
\end{itemize}
\end{lemma}
\begin{proof}
For a given $\bar{x}^k$, we have $\bar{y}^k_i = \psi_i(\bar{x}^k)$ and $F_i(x) \preceq G_i(x; \bar{y}^k_i) \preceq 0$ for $i=1,\dots, m$. Thus if
$x\in\mathcal{F}(\bar{x}^k)$ then $x\in\mathcal{F}$, the statement a) holds. 
Consequently, the sequence $\{\bar{x}^k\}$ is feasible to \eqref{eq:NSDP} which is indeed the statement b).
Since $\bar{x}^{k+1}$ is a solution of \ref{eq:convx_subprob}, it shows that $\bar{x}^{k+1}\in\mathcal{F}(\bar{x}^k)$. 
Now, we have to show it belongs to $\mathcal{F}(\bar{x}^{k+1})$. 
Indeed, since $G_i(\bar{x}^{k+1}, \bar{y}_i^{k+1}) = F_i(\bar{x}^{k+1}) \preceq 0$ by Definition \ref{def:over_relaxation} for all $i=1,\dots, m$, we conclude
$\bar{x}^{k+1}\in\mathcal{F}(\bar{x}^{k+1})$. The statement c) is proved.
Finally, we prove d). 
Since $\bar{x}^{k+1}$ is the optimal solution of \ref{eq:convx_subprob}, we have $f(\bar{x}^{k+1}) +
\frac{1}{2}\norm{\bar{x}^{k+1}-\bar{x}^k}^2_{Q_k} \leq f(x) + \frac{1}{2}(x-x^k)^TQ_k(x-x^k) - \frac{\rho_f}{2}\norm{x-\bar{x}^{k+1}}^2$ for all
$x\in\mathcal{F}(\bar{x}^k)$.
However, we have $\bar{x}^k\in\mathcal{F}(\bar{x}^k)$ due to c). By substituting $x = \bar{x}^k$ in the previous inequality we obtain the estimate d).
\end{proof}

Now, we denote by $\mathcal{L}_f(\alpha) := \left\{x\in\mathcal{F}~|~ f(x) \leq \alpha\right\}$ the lower level set of the objective function. 
Let us assume that $G_i(\cdot;y)$ is continuously differentiable in $\mathcal{L}_f(f(\bar{x}^0))$ for any $y$. 
We say that the \textit{Robinson qualification} condition for \ref{eq:convx_subprob} holds at $\bar{x}$ if $0\in\mathrm{int}(G_i(\bar{x};\bar{y}_i^k)+
D_xG_i(\bar{x};\bar{y}_i^k)(\Omega-\bar{x}) + \mathcal{S}^p_{+})$ for $i=1,\dots, m$. In order to prove the convergence of Algorithm \ref{alg:A1}, we require
the following assumption.

\begin{assumption}\label{as:A2}
The set of KKT points of \eqref{eq:NSDP} is nonempty. For a given $y$, the matrix-valued mappings $G_i(\cdot;y)$ are continuously differentiable on
$\mathcal{L}_f(f(\bar{x}^0))$. The convex problem \ref{eq:convx_subprob} is solvable and the Robinson qualification condition holds at its solutions.
\end{assumption}
We note that if Algorithm 1 is terminated at the iteration $k$ such that $\bar{x}^k = \bar{x}^{k+1}$ then $\bar{x}^k$ is a stationary point of \eqref{eq:NSDP}.

\begin{theorem}\label{th:convergence}
Suppose that Assumptions A.\ref{as:A1} and A.\ref{as:A2} are satisfied. Suppose further that the lower level set $\mathcal{L}_f(f(\bar{x}^0))$ is bounded. 
Let $\{(\bar{x}^k, \bar{W}^k)\}_{k\geq 1}$ be an infinite sequence generated by Algorithm \ref{alg:A1} starting from $\bar{x}^0\in\mathcal{F}^0$.
Assume that $\lambda_{\max}(Q_k) \leq M <+\infty$.
Then if either $f$ is strongly convex or $\lambda_{\min}(Q_k) \geq \rho > 0$ for $k\geq 0$ then every accumulation point $(\bar{x}^{*}, \bar{W}^{*})$ of
$\{(\bar{x}^k, \bar{W}^k)\}$ is a KKT point of \eqref{eq:NSDP}.
Moreover, if the set of the KKT points of \eqref{eq:NSDP} is finite then the whole sequence $\{ (\bar{x}^k, \bar{W}^k) \}$ converges to a KKT point of
\eqref{eq:NSDP}.
\end{theorem}

\begin{proof}
First, we show that the solution mapping $\mathcal{S}(\bar{x}^k, Q_k)$ is \textit{closed}. 
Indeed, by Assumption A.\ref{as:A2}, \ref{eq:convx_subprob} is feasible. Moreover, it is strongly convex. Hence, $\mathcal{S}(\bar{x}^k, Q_k) =
\left\{\bar{x}^{k+1}\right\}$, which is obviously closed.
The remaining conclusions of the theorem can be proved similarly as \cite[Theorem 3.2.]{TranDinh2011} by using Zangwill's convergence theorem
\cite[p. 91]{Zangwill1969} of which we omit the details here.
\end{proof}

\begin{remark}\label{rm:conclusions}
Note that the assumptions used in the proof of the closedness of the solution mapping $\mathcal{S}(\cdot)$ in Theorem \ref{th:convergence} are weaker than the
ones used in \cite[Theorem 3.2.]{TranDinh2011}. 
\end{remark}

\section{Applications to robust controller design}\label{sec:app}
In this section, we present some applications of Algorithm \ref{alg:A1} for solving several classes of optimization problems arising in static output feedback
controller design. Typically, these problems are related to the following linear, time-invariant (LTI) system of the form:
\begin{equation}\label{eq:LTI}
\left\{\begin{array}{cl}
&\dot{x} = Ax + B_1w + Bu,\\
&z = C_1x + D_{11}w + D_{12}u,\\
&y = Cx + D_{21}w,
\end{array}\right.
\end{equation}
where $x\in\mathbb{R}^{n}$ is the state vector, $w\in\mathbb{R}^{n_w}$ is the performance input, $u\in\mathbb{R}^{n_u}$ is the input vector,
$z\in\mathbb{R}^{n_z}$ is the performance output, $y\in\mathbb{R}^{n_y}$ is the physical output vector, $A\in\mathbb{R}^{n\times n}$ is state matrix,
$B\in\mathbb{R}^{n\times n_u}$ is input matrix and $C\in\mathbb{R}^{n_y\times n}$ is the output matrix.
By using a static feedback controller of the form $u = Fy$ with $F\in\mathbb{R}^{n_u\times n_y}$, we can write the closed-loop system as follows:
\begin{equation}\label{eq:ss_LTI}
\left\{\begin{array}{cl}
\dot{x}_F = A_Fx_F + B_Fw,\\
z = C_Fx_F + D_Fw.
\end{array}\right.
\end{equation}
The stabilization, $\h_2$, $\h_\infty$ optimization and other control problems of the LTI system can be formulated as an optimization problem with BMI
constraints.
We only use the psd-convex overestimate of a bilinear form in \textit{Example 3} to show that Algorithm \ref{alg:A1} can be applied to solving many
problems ins static state/output feedback controller design such as:
\begin{itemize}
\item[1.] Sparse linear static output feedback controller design;
\item[2.] Spectral abscissa and pseudospectral abscissa optimization;
\item[3.] $\mathcal{H}_2$ optimization;
\item[4.] $\mathcal{H}_{\infty}$ optimization;
\item[5.] and mixed $\mathcal{H}_2/\mathcal{H}_{\infty}$ synthesis.
\end{itemize}
These problems possess at least one BMI constraint of the from $\tilde{B}_I(X, Y, Z)\preceq 0$, where $\tilde{B}_I(X, Y, Z) := X^TY + Y^TX + \mathcal{A}(Z)$,
where $X, Y$ and $Z$ are matrix variables and $\mathcal{A}$ is a affine operator of matrix variable $Z$. 
By means of \textit{Example 3}, we can approximate the bilinear term $X^TY + Y^TX$ by its psd-convex overestimate. Then using Schur's
complement to transform the constraint $G_i(x;x^k) \preceq 0$ of the subproblem \ref{eq:convx_subprob} into an LMI constraint \cite{TranDinh2011}. 
Note that Algorithm \ref{alg:A1} requires an interior starting point $x^0\in\mathcal{F}^0$. In this work, we apply the procedures proposed in
\cite{TranDinh2011} to find such a point.
Now, we summary the whole procedure applying to solve the optimization problems with BMI constraints as follows:

\begin{scheme}\label{scheme:A1}{~}\\
\textit{Step 1. } Find a psd-convex overestimate $G_i(x;y)$ of $F_i(x)$ w.r.t. the parameterization $y = \psi_i(x)$ for $i=1,\dots,m$ (see \textit{Example
1}).\\
\textit{Step 2. } Find a starting point $\bar{x}^0\in\mathcal{F}^0$ (see \cite{TranDinh2011}).\\
\textit{Step 3. } For  a given $\bar{x}^k$, form the convex semidefinite programming problem \ref{eq:convx_subprob} and reformulate it as an optimization with
LMI constraints.\\
\textit{Step 4. } Apply Algorithm \ref{alg:A1} with an SDP solver to solve the given problem.
\end{scheme}
Now, we test Algorithm \ref{alg:A1} for three problems via numerical examples by using the data from  the COMP$\textrm{l}_\textrm{e}$ib library
\cite{Leibfritz2003}. 
All the implementations are done in Matlab 7.8.0 (R2009a) running on a Laptop Intel(R) Core(TM)i7 Q740 \@ 1.73GHz and 4Gb RAM.
We use the YALMIP package \cite{Lofberg2004} as a modeling language and SeDuMi 1.1 as a SDP solver \cite{Sturm1999} to solve the LMI optimization problems
arising in Algorithm \ref{alg:A1} at the initial phase (Phase 1) and the subproblem \ref{eq:convx_subprob}. 
The code is available at \url{http://www.kuleuven.be/optec/software/BMIsolver}. 
We also compare the performance of Algorithm \ref{alg:A1} and the convex-concave decomposition method (CCDM) proposed in \cite{TranDinh2011} in the first
example, i.e. the spectral abscissa optimization problem. 
In the second example, we compare the $\h_{\infty}$-norm computed by Algorithm \ref{alg:A1} and the one provided by HIFOO \cite{Gumussoy09} and PENBMI
\cite{Henrion2005}. 
The last example is the mixed $H_2/H_{\infty}$ synthesis optimization problem which we compare between two values of the $H_2$-norm level.

\subsection{Spectral abscissa optimization}\label{subsec:SOF_design}
We consider an optimization problem with BMI constraint by optimizing the spectral abscissa of the closed-loop system $\dot{x} = (A+BFC)x$ as
\cite{Burke2002,Leibfritz2004}:
\begin{equation}\label{eq:BMI_p2}
\left\{\begin{array}{cl}
\displaystyle
\max_{P,F,\beta} &{\!\!\!} \beta\\
\textrm{s.t.} &{\!\!\!}(A \!+\! BFC)^TP \!+\! P(A \!+\! BFC) \!+\! 2\beta P \prec 0, \\
& P = P^T, ~P \succ 0.
\end{array}\right.
\end{equation}
Here, matrices $A\in\mathbb{R}^{n\times n}$, $B\in\mathbb{R}^{n\times n_u}$ and $C\in\mathbb{R}^{n_y\times n}$ are given. Matrices $P\in\mathbb{R}^{n\times n}$
and $F\in\mathbb{R}^{n_u\times n_y}$ and the scalar $\beta$ are considered as variables.
If the optimal value of \eqref{eq:BMI_p2} is strictly positive then the closed-loop feedback controller $u = Fy$ stabilizes the linear system $\dot{x} =
(A+BFC)x$.

By introducing an intermediate variable $A_F := A + BFC + \beta I$, the BMI constraint in the second line of \eqref{eq:BMI_p2} can be written $A_F^TP +
P^TA_F \prec 0$.
Now, by applying Scheme \ref{scheme:A1} one can solve the problem \eqref{eq:BMI_p2} by exploiting the Sedumi SDP solver \cite{Sturm1999}. 
In order to obtain a strictly descent direction, we regularize the subproblem \ref{eq:convx_subprob} by adding quadratic terms: $\rho_F\norm{F-F^k}_F^2 +
\rho_P\norm{P-P^k}^2_F +
\rho_f|\beta-\beta_k|^2$, where $\rho_F = \rho_P = \rho_f = 10^{-3}$.
Algorithm \ref{alg:A1} is terminated if one of the following conditions is satisfied:
\begin{itemize}
\item the subproblem \ref{eq:convx_subprob} encounters a numerical problem;
\item $\norm{\bar{x}^{k+1}-\bar{x}^k}_{\infty}/(\norm{\bar{x}^k}_{\infty}+1)\leq 10^{-3}$;
\item the maximum number of iterations, $K_{\max}$, is reached;
\item or the objective function of \eqref{eq:NSDP} is not significantly improved after two successive iterations, i.e. $|f^{k+1}-f^k| \leq 10^{-4}(1+|f^k|)$ for
some $k=\bar{k}$
and $k=\bar{k}+1$, where $f^k := f(\bar{x}^k)$.
\end{itemize}
We test Algorithm \ref{alg:A1} for several problems in COMP$\textrm{l}_\textrm{e}$ib and compare our results with the ones reported by the
\textit{convex-concave decomposition method} (CCDM) in \cite{TranDinh2011}.

\begin{center}
\begin{table}[!ht]
\vskip-0.45cm
\renewcommand{\arraystretch}{0.7}
\begin{scriptsize}
\caption{Computational results for \eqref{eq:BMI_p2} in COMP$\textrm{l}_{\textrm{e}}$ib}\label{tb:exam2}
\newcommand{\cell}[1]{{\!\!}#1{\!\!}}
\newcommand{\cellbf}[1]{{\!\!}\textbf{#1}{\!\!}}
\begin{tabular}{|l|r|r|r|r|r|r|r|r|r|}\hline
\multicolumn{2}{|c|}{Problem} & \multicolumn{3}{c|}{{\!\!\!}Convex-Concave Decom.{\!\!\!\!}} & \multicolumn{3}{c|}{{\!\!\!\!}Inner Convex
App.{\!\!\!\!}} \\ \hline
\!\!\texttt{Name}\!\!\! & \!\!\!$\alpha_0(A)$\!\!\!&\!\!\texttt{CCDM}~\!\!&\!\!\!\!\texttt{iter}\!\!&\!\!\texttt{time[s]}\!\!& \!\!$\alpha_0(A_F)$\!\! &
\!\!\!\!\texttt{Iter}\!\! &
\!\!\!\!\!\!\texttt{time[s]}\!\!\! \\ \hline
\cell{AC1}  & \cell{0.000} & \cellbf{-0.8644} & \cell{  62} & \cell{23.580} & \cellbf{-0.7814} & \cell{  55} & \cell{19.510} \\ \hline
\cell{AC4}  & \cell{2.579} & \cellbf{-0.0500} & \cell{  14} & \cell{ 6.060} & \cellbf{-0.0500} & \cell{  14} & \cell{ 4.380} \\ \hline
\cell{\textbf{AC5}$^{\textbf{a}}$}
            & \cell{0.999} & \cellbf{-0.7389} & \cell{  28} & \cell{10.200} & \cellbf{-0.7389} & \cell{  37} & \cell{12.030} \\ \hline
\cell{AC7}  & \cell{0.172} & \cellbf{-0.0766} & \cell{ 200} & \cell{95.830} & \cellbf{-0.0502} & \cell{  90} & \cell{80.710} \\ \hline
\cell{AC8}  & \cell{0.012} & \cellbf{-0.0755} & \cell{  24} & \cell{12.110} & \cellbf{-0.0640} & \cell{  40} & \cell{32.340} \\ \hline
\cell{AC9}  & \cell{0.012} & \cellbf{-0.4053} & \cell{ 100} & \cell{55.460} & \cellbf{-0.3926} & \cell{ 200} & \cell{217.230} \\ \hline
\cell{AC11} & \cell{5.451} & \cellbf{-5.5960} & \cell{ 200} & \cell{81.230} & \cellbf{-3.1573} & \cell{ 181} & \cell{73.660} \\ \hline
\cell{AC12} & \cell{0.580} & \cellbf{-0.5890} & \cell{ 200} & \cell{61.920} & \cellbf{-0.2948} & \cell{ 200} & \cell{71.200} \\ \hline
\cell{HE1}  & \cell{0.276} & \cellbf{-0.2241} & \cell{ 200} & \cell{56.890} & \cellbf{-0.2134} & \cell{ 200} & \cell{58.580} \\ \hline
\cell{HE3}  & \cell{0.087} & \cellbf{-0.9936} & \cell{ 200} & \cell{98.730} & \cellbf{-0.8380} & \cell{  57} & \cell{54.720} \\ \hline
\cell{HE4}  & \cell{0.234} & \cellbf{-0.8647} & \cell{  63} & \cell{27.620} & \cellbf{-0.8375} & \cell{  88} & \cell{70.770} \\ \hline
\cell{HE5}  & \cell{0.234} & \cellbf{-0.1115} & \cell{ 200} & \cell{86.550} & \cellbf{-0.0609} & \cell{ 200} & \cell{181.470} \\ \hline
\cell{HE6}  & \cell{0.234} & \cellbf{-0.0050} & \cell{  12} & \cell{29.580} & \cellbf{-0.0050} & \cell{  18} & \cell{106.840} \\ \hline
\cell{REA1} & \cell{1.991} & \cellbf{-4.2792} & \cell{ 200} & \cell{70.370} & \cellbf{-2.8932} & \cell{ 200} & \cell{74.560} \\ \hline
\cell{REA2} & \cell{2.011} & \cellbf{-2.1778} & \cell{  40} & \cell{13.360} & \cellbf{-1.9514} & \cell{  43} & \cell{13.120} \\ \hline
\cell{REA3} & \cell{0.000} & \cellbf{-0.0207} & \cell{ 200} & \cell{267.160}& \cellbf{-0.0207} & \cell{ 161} & \cell{311.490} \\ \hline
\cell{DIS2} & \cell{1.675} & \cellbf{-8.4540} & \cell{  28} & \cell{ 9.430} & \cellbf{-8.3419} & \cell{  44} & \cell{12.600} \\ \hline
\cell{DIS4} & \cell{1.442} & \cellbf{-8.2729} & \cell{  95} & \cell{40.200} & \cellbf{-5.4467} & \cell{  89} & \cell{40.120} \\ \hline
\cell{WEC1} & \cell{0.008} & \cellbf{-0.8972} & \cell{ 200} & \cell{121.300}& \cellbf{-0.8568} & \cell{  68} & \cell{76.000} \\ \hline
\cell{IH}   & \cell{0.000} & \cellbf{-0.5000} & \cell{   7} & \cell{23.670} & \cellbf{-0.5000} & \cell{  11} & \cell{82.730} \\ \hline
\cell{CSE1} & \cell{0.000} & \cellbf{-0.3093} & \cell{  81} & \cell{219.910}& \cellbf{-0.2949} & \cell{ 200} & \cell{1815.400} \\ \hline
\cell{TF1}  & \cell{0.000} & \cellbf{-0.1598} & \cell{  87} & \cell{34.960} & \cellbf{-0.0704} & \cell{ 200} & \cell{154.430} \\ \hline
\cell{TF2}  & \cell{0.000} & \cellbf{-0.0000} & \cell{   8} & \cell{ 4.220} & \cellbf{-0.0000} & \cell{  12} & \cell{10.130} \\ \hline
\cell{TF3}  & \cell{0.000} & \cellbf{-0.0031} & \cell{  93} & \cell{35.000} & \cellbf{-0.0032} & \cell{  95} & \cell{70.980} \\ \hline
\cell{NN1}  & \cell{3.606} & \cellbf{-1.5574} & \cell{ 200} & \cell{57.370} & \cellbf{0.1769}  & \cell{ 200} & \cell{59.230} \\ \hline
\cell{\textbf{NN5}$^{\textbf{a}}$}
            & \cell{0.420} & \cellbf{-0.0722} & \cell{ 200} & \cell{79.210} & \cellbf{-0.0490} & \cell{ 200} & \cell{154.160} \\ \hline
\cell{NN9}  & \cell{3.281} & \cellbf{-0.0279} & \cell{  33} & \cell{11.880} & \cellbf{0.0991} & \cell{  44} & \cell{13.860} \\ \hline
\cell{NN13} & \cell{1.945} & \cellbf{-3.4412} & \cell{ 181} & \cell{64.500} & \cellbf{-0.2783} & \cell{  32} & \cell{12.430} \\ \hline
\cell{NN15} & \cell{0.000} & \cellbf{-1.0424} & \cell{ 200} & \cell{58.440} & \cellbf{-1.0409} & \cell{ 200} & \cell{60.930} \\ \hline
\cell{NN17} & \cell{1.170} & \cellbf{-0.6008} & \cell{  99} & \cell{27.190} & \cellbf{-0.5991} & \cell{ 132} & \cell{34.820} \\ \hline
\end{tabular}
\end{scriptsize}
\vskip -0.45cm
\end{table}
\end{center}
The numerical results and the performances of two algorithms are reported in Table \ref{tb:exam2}. Here, we initialize both algorithms  with the same initial
guess $F^0 = 0$.

The notation in Table \ref{tb:exam2} consists of: \texttt{Name} is the name of problems, $\alpha_0(A)$, $\alpha_0(A_F)$ are the maximum real part of the
eigenvalues of the open-loop and closed-loop matrices $A$, $A_F$, respectively; \texttt{iter} is the number of iterations, \texttt{time[s]} is the CPU time in
seconds. 
Both methods, Algorithm \ref{alg:A1} and CCDM fail or make only slow progress towards a local solution with $6$ problems: AC18, DIS5, PAS, NN6, NN7, NN12 in
COMP$\textrm{l}_\textrm{e}$ib. 
Problems AC5 and NN5 are initialized with a different matrix $F^0$ to avoid numerical problems.
The numerical results show that the performances of both methods are quite similar for the majority of problems.

Note that Algorithm \ref{alg:A1} as well as the algorithm in \cite{TranDinh2011} are local optimization methods which only find a local minimizer and these
solutions may not be the same. 

\subsection{$\mathcal{H}_{\infty}$ optimization: BMI formulation}\label{subsec:Hinf_prob}
Next, we apply Algorithm \ref{alg:A1} to solve the optimization with BMI constraints arising in $\mathcal{H}_{\infty}$ optimization of the linear system
\eqref{eq:LTI}. 
In this example we assume that $D_{21} = 0$, this problem is reformulated as the following optimization problem with BMI constraints \cite{Leibfritz2003}:
\begin{equation}\label{eq:BMI_Hinf_prob}
\begin{array}{cl}
\displaystyle\min_{F, X, \gamma} &\gamma \\
\textrm{s.t.} & \begin{bmatrix}A_F^TX + XA_F &  XB_1 & C_F^T\\ B_1^TX & -\gamma I_w & D_{11}^T\\ C_F & D_{11} & -\gamma I_z\end{bmatrix} \prec
0,\\
& X \succ 0, ~ \gamma > 0.
\end{array}
\end{equation}
Here, as before, we define $A_F := A + BFC$ and $C_F := C_1 + D_{12}FC$.
The bilinear matrix term $A_F^TX + XA_F$ at the top-corner of the first constraint can be approximated by the form of $\mathcal{Q}_Q$ defined in
\eqref{eq:G_Q}. 
Therefore, we can use this psd-convex overestimate to approximate the problem \eqref{eq:BMI_Hinf_prob} by a sequence of the convex subproblems of
the form \ref{eq:convx_subprob}. Then we transform the subproblem into a standard SDP problem that can be solve by a standard SDP solver thanks to
Schur's complement \cite{Bernstein2005,TranDinh2011}.

To determine a starting point, we perform the heuristic procedure called \textit{Phase 1} proposed in \cite{TranDinh2011} which is terminated after a finite
number of iterations.
In this example, we also test Algorithm \ref{alg:A1} for several problems in COMP$\textrm{l}_{\textrm{e}}$ib using the same parameters and the stopping
criterion as in the previous subsection.  The computational results are shown in Table \ref{tbl:H_inf_problems}.
The numerical results computed by HIFOO and PENBMI are also included in Table \ref{tbl:H_inf_problems}. 

Here, three last columns are the results and the performances of our method, the columns HIFOO and PENBMI indicate the $\mathcal{H}_{\infty}$-norm of the
closed-loop system for the static output feedback controller given by HIFOO and PENBMI, respectively. 
We can see from Table \ref{tbl:H_inf_problems} that the optimal values reported by Algorithm \ref{alg:A1} and HIFOO are almost similar for many
problems whereas in general PENBMI has difficulties in finding a feasible solution.

\begin{center}
\begin{table}[!ht]
\vskip-0.3cm
\begin{scriptsize}
\renewcommand{\arraystretch}{0.7}
\caption{$\mathcal{H}_{\infty}$ synthesis benchmarks on  COMP$\textrm{l}_{\textrm{e}}$ib plants}
\label{tbl:H_inf_problems}
\newcommand{\cell}[1]{{\!\!\!\!}#1{\!\!\!}}
\newcommand{\cellbf}[1]{{\!\!\!\!}\textbf{#1}{\!\!\!}}
\begin{tabular}{|l|c|c|c|c|r|r|r|r|r|r|}\hline
\multicolumn{6}{|c|}{Problem information} & \multicolumn{2}{c|}{Other Results, $\mathcal{H}_\infty$} & \multicolumn{3}{|c|}{\!\!{Results and Performances 
}\!\!} \\ \hline
\!\!\texttt{Name}\!\!\! & \!\!\!$n_x$\!\!\! & \!\!\!$n_y$\!\!\! & \!\!\!$n_u$\!\!\! & \!\!\!$n_z$\!\!\! & \!\!\!$n_w$\!\!\! & 
\!\!HIFOO\!\!\! & \!\!PENBMI\!\! & \!\!\!$\mathcal{H}_\infty$\!\!\! & \!\!\!\!\!\texttt{iter}\!\!\! & \!\!\!\!\!\!\!\texttt{time\![s]}\!\!\!\!\! \\ \hline
\cell{AC2}       & \cell{5} & \cell{3} & \cell{3} & \cell{5} & \cell{3} & \cellbf{0.1115} & \cellbf{-} & \cellbf{0.1174} & \cell{ 120} & \cell{91.560} \\ \hline
\cell{AC3}       & \cell{5} & \cell{4} & \cell{2} & \cell{5} & \cell{5} & \cellbf{4.7021} & \cellbf{-} & \cellbf{3.5053} & \cell{ 267} & \cell{193.940}\\ \hline
\cell{AC6}       & \cell{7} & \cell{4} & \cell{2} & \cell{7} & \cell{7} & \cellbf{4.1140} & \cellbf{-} & \cellbf{4.1954} & \cell{167}  & \cell{138.570} \\
\hline
\cell{AC7}       & \cell{9} & \cell{2} & \cell{1} & \cell{1} & \cell{4} & \cellbf{0.0651} & \cellbf{0.3810} & \cellbf{0.0339} & \cell{ 300} & \cell{276.310} \\
\hline
\cell{AC8}       & \cell{9} & \cell{5} & \cell{1} & \cell{2} & \cell{10}& \cellbf{2.0050} & \cellbf{-} &  \cellbf{4.5463} & \cell{ 224} & \cell{230.990} \\
\hline
\cell{AC11$^{b}$}& \cell{5} & \cell{4} & \cell{2} & \cell{5} & \cell{5} &  \cellbf{3.5603} & \cellbf{-} & \cellbf{3.4924} & \cell{ 300} & \cell{255.620} \\
\hline
\cell{AC15}      & \cell{4} & \cell{3} & \cell{2} & \cell{6} & \cell{4} & \cellbf{15.2074}& \cellbf{427.4106} & \cellbf{15.2036} & \cell{ 153} & \cell{130.660}
\\ \hline
\cell{AC16}      & \cell{4} & \cell{4} & \cell{2} & \cell{6} & \cell{4} & \cellbf{15.4969}& \cellbf{-} & \cellbf{15.0433} & \cell{ 267} & \cell{201.360} \\
\hline
\cell{AC17}      & \cell{4} & \cell{2} & \cell{1} & \cell{4} & \cell{4} & \cellbf{6.6124} & \cellbf{-} &  \cellbf{6.6571} & \cell{ 192} & \cell{64.880} \\
\hline
\cell{HE1$^{b}$} & \cell{4} & \cell{1} & \cell{2} & \cell{2} & \cell{2} & \cellbf{0.1540} & \cellbf{1.5258} & \cellbf{0.2188} & \cell{ 300} & \cell{97.760} \\
\hline
\cell{HE3}       & \cell{8} & \cell{6} & \cell{4} & \cell{10}& \cell{1} & \cellbf{0.8545} & \cellbf{1.6843} &  \cellbf{0.8640} & \cell{  15} & \cell{16.320} \\
\hline
\cell{HE5$^{b}$} & \cell{8} & \cell{2} & \cell{4} & \cell{4} & \cell{3} & \cellbf{8.8952} & \cellbf{-} &  \cellbf{36.3330} & \cell{ 154} & \cell{208.680} \\
\hline
\cell{REA1}      & \cell{4} & \cell{3} & \cell{2} & \cell{4} & \cell{4} & \cellbf{0.8975} & \cellbf{-} &  \cellbf{0.8815} & \cell{ 183} & \cell{67.790} \\
\hline
\cell{REA2$^{b}$}& \cell{4} & \cell{2} & \cell{2} & \cell{4} & \cell{4} & \cellbf{1.1881} & \cellbf{-} & \cellbf{1.4444} & \cell{ 300} & \cell{109.430} \\
\hline
\cell{REA3}      & \cell{12}& \cell{3} & \cell{1} & \cell{12}& \cell{12}& \cellbf{74.2513}& \cellbf{74.4460} & \cellbf{75.0634} & \cell{   2} & \cell{137.120}
\\ \hline
\cell{DIS1}      & \cell{8} & \cell{4} & \cell{4} & \cell{8} & \cell{1} & \cellbf{4.1716} & \cellbf{-} & \cellbf{4.2041} & \cell{ 129} & \cell{110.330} \\
\hline
\cell{DIS2}      & \cell{3} & \cell{2} & \cell{2} & \cell{3} & \cell{3} & \cellbf{1.0548} & \cellbf{1.7423} & \cellbf{1.1570} & \cell{  78} & \cell{28.330}
\\\hline
\cell{DIS3}      & \cell{5} & \cell{3} & \cell{3} & \cell{2} & \cell{3} & \cellbf{1.0816} & \cellbf{-} &  \cellbf{1.1701} & \cell{ 219} & \cell{160.680} \\
\hline
\cell{DIS4}      & \cell{6} & \cell{6} & \cell{4} & \cell{6} & \cell{6} & \cellbf{0.7465} & \cellbf{-} &  \cellbf{0.7532} & \cell{ 171} & \cell{126.940} \\
\hline
\cell{TG1$^{b}$} & \cell{10}& \cell{2} & \cell{2} & \cell{10}& \cell{10}& \cellbf{12.8462}& \cellbf{-} &  \cellbf{12.9461} & \cell{  64} & \cell{264.050} \\
\hline
\cell{AGS}       & \cell{12}& \cell{2} & \cell{2} & \cell{12}& \cell{12}& \cellbf{8.1732} & \cellbf{188.0315} &  \cellbf{8.1733} & \cell{  41} & \cell{160.880}
\\ \hline
\cell{WEC2}      & \cell{10}& \cell{4} & \cell{3} & \cell{10}& \cell{10}& \cellbf{4.2726} & \cellbf{32.9935} &  \cellbf{8.8809} & \cell{ 300} & \cell{1341.760}
\\ \hline
\cell{WEC3}      & \cell{10}& \cell{4} & \cell{3} & \cell{10}& \cell{10}& \cellbf{4.4497} & \cellbf{200.1467} & \cellbf{7.8215} & \cell{ 225} & \cell{875.100}
\\ \hline
\cell{BDT1}      & \cell{11}& \cell{3} & \cell{3} & \cell{6} & \cell{1} & \cellbf{0.2664} & \cellbf{-} & \cellbf{0.8544} & \cell{   3} & \cell{ 5.290} \\ \hline
\cell{MFP}       & \cell{4} & \cell{2} & \cell{3} & \cell{4} & \cell{4} & \cellbf{31.5899}& \cellbf{-} & \cellbf{31.6388} & \cell{ 300} & \cell{100.660} \\
\hline
\cell{IH}        & \cell{21}& \cell{10}& \cell{11}& \cell{11}& \cell{21}& \cellbf{1.9797} & \cellbf{-} &
\cellbf{1.1861} & \cell{ 210} & \cell{2782.880} \\ \hline
\cell{CSE1}      & \cell{20}& \cell{10}& \cell{2} & \cell{12}& \cell{1} & \cellbf{0.0201} & \cellbf{-} &  \cellbf{0.0219} & \cell{   3} & \cell{39.330} \\
\hline
\cell{PSM}       & \cell{7} & \cell{3} & \cell{2} & \cell{5} & \cell{2} & \cellbf{0.9202} & \cellbf{-} & \cellbf{0.9266} & \cell{ 153} & \cell{104.170} \\
\hline
\cell{EB1}       & \cell{10}& \cell{1} & \cell{1} & \cell{2} & \cell{2} & \cellbf{3.1225} & \cellbf{39.9526} & \cellbf{2.0532} & \cell{ 300} & \cell{299.380} \\
\hline
\cell{EB2}       & \cell{10}& \cell{1} & \cell{1} & \cell{2} & \cell{2} & \cellbf{2.0201} & \cellbf{39.9547} &  \cellbf{0.8150} & \cell{ 120} & \cell{103.400}
\\ \hline
\cell{EB3}       & \cell{10}& \cell{1} & \cell{1} & \cell{2} & \cell{2} & \cellbf{2.0575} & \cellbf{3995311.0743} &  \cellbf{0.8157} & \cell{ 117} &
\cell{116.390} \\ \hline
\cell{NN2}       & \cell{2} & \cell{1} & \cell{1} & \cell{2} & \cell{2} & \cellbf{2.2216} & \cellbf{-} & \cellbf{2.2216} & \cell{  15} & \cell{ 7.070} \\ \hline
\cell{NN4}       & \cell{4} & \cell{3} & \cell{2} & \cell{4} & \cell{4} & \cellbf{1.3627} & \cellbf{-} & \cellbf{1.3884} & \cell{ 204} & \cell{70.200} \\
\hline
\cell{NN8}       & \cell{3} & \cell{2} & \cell{2} & \cell{3} & \cell{3} & \cellbf{2.8871} & \cellbf{78281181.1490} & \cellbf{2.9522} & \cell{ 240} &
\cell{84.510} \\ \hline 
\cell{NN11$^{b}$}& \cell{16}& \cell{5} & \cell{3} & \cell{3} & \cell{3} & \cellbf{0.1037} & \cellbf{-} &  \cellbf{0.1596} & \cell{  15} & \cell{86.770} \\
\hline
\cell{NN15}      & \cell{3} & \cell{2} & \cell{2} & \cell{4} & \cell{1} & \cellbf{0.1039} & \cellbf{-} & \cellbf{0.1201} & \cell{   6} & \cell{ 4.000} \\ \hline
\cell{NN16}      & \cell{8} & \cell{4} & \cell{4} & \cell{4} & \cell{8} & \cellbf{0.9557} & \cellbf{-} & \cellbf{0.9699} & \cell{  36} & \cell{32.200} \\ \hline
\cell{NN17}      & \cell{3} & \cell{1} & \cell{2} & \cell{2} & \cell{1} & \cellbf{11.2182}& \cellbf{-} & \cellbf{11.2538} & \cell{ 270} & \cell{81.480} \\
\hline
\end{tabular}
\end{scriptsize}
\vskip -0.3cm
\end{table}
\end{center}

\subsection{$\mathcal{H}_2/\mathcal{H}_{\infty}$ optimization: BMI formulation}\label{subsec:H2Hinf_prob}
Motivated from the $\mathcal{H}_{\infty}$ optimization problem, in this example we consider the mixed $\mathcal{H}_2/\mathcal{H}_{\infty}$
synthesis optimization problem.
Let us assume that $D_{11} = 0$, $D_{21} = 0$ and the performance output $z$ is divided in two components, $z_1$ and $z_2$.
Then the linear system \eqref{eq:LTI} becomes:
\begin{equation}\label{eq:mixed_LTI_z1z2}
\left\{\begin{array}{cl}
&\dot{x} = Ax + B_1w + Bu,\\
&z_1 = C_1^{z_1}x + D_{12}^{z_1}u,\\
&z_2 = C_1^{z_2}x + D_{12}^{z_2}u,\\
&y   = Cx.
\end{array}\right.
\end{equation}
The mixed $\mathcal{H}_2/\mathcal{H}_{\infty}$ control problem is to find a static output feedback gain $F$ such that, for $u = Fy$, the $\mathcal{H}_2$-norm of
the closed loop from $w$ to $z_2$ is
minimized, while the $\mathcal{H}_{\infty}$-norm from $w$ to $z_1$ is less than some imposed level $\gamma$ \cite{Boyd1994,Leibfritz2004,Ostertag2008}.

This problem leads to the following optimization problem with BMI constraints \cite{Ostertag2008}:
\begin{equation}\label{eq:H2Hinf_prob}
\begin{array}{cl}
\displaystyle\min_{F, P_1,P_2,Z}  &{\!\!\!\!} \textrm{trace}(Z)\\
\textrm{s.t.} 
&{\!\!\!\!} \begin{bmatrix}A_F^T\!P_1 \!+\! P_1A_F \!+\! (C_{F}^{z_1})^T\!C_{F}^{z_1} {\!\!}&{\!\!} P_1B_1\\ B_1^TP_1 {\!\!}&{\!\!} -\gamma^2I\end{bmatrix}\prec
0,\\
&{\!\!\!\!}\begin{bmatrix}A_F^TP_2 + P_2A_F & P_2B_1\\B_1^TP_2 & -I\end{bmatrix}\prec 0,\\
&{\!\!\!\!}\begin{bmatrix}P_2 & (C_{F}^{z_2})^T\\C_{F}^{z_2} & Z \end{bmatrix}\succ 0, ~ P_1 \succ 0, ~P_2\succ 0,
\end{array}
\end{equation}
where $A_F := A + BFC$, $C_{F}^{z_1} := C_1^{z_1} + D_{12}^{z_1}FC$ and $C_{F}^{z_2} :=   C_1^{z_2} + D_{12}^{z_2}FC$.
Note that if $C = I_{n_x}$, the identity matrix, then this problem becomes a mixed $\mathcal{H}_2/\mathcal{H}_{\infty}$ of static state feedback design problem
considered in \cite{Ostertag2008}.

Now, we implement Algorithm \ref{alg:A1} for solving the problem \eqref{eq:H2Hinf_prob}.
As before, we use a procedure proposed in \cite{TranDinh2011} to determine a starting point for Algorithm \ref{alg:A1}.
We test the algorithm described above for several problems in COMP$\textrm{l}_{\textrm{e}}$ib with the level values $\gamma = 4$ and $\gamma = 10$. In this
test, we assume that the output signals $z_1\equiv z_2$. Thus we have $C_1^{z_1} = C_1^{z_2} = C_1$ and $D_{12}^{z_1} = D_{12}^{z_2} = D_{12}$.
The parameters and the stopping criterion of the algorithm are chosen as in the $\mathrm{H}_{\infty}$ problem. The computational results are reported in Table
\ref{tb:H2Hinf_problems} with $\gamma = 4$ and $\gamma = 10$.
\begin{table}[!ht]
\begin{center}
\renewcommand{\arraystretch}{0.7}	
\begin{scriptsize}
\caption{$\mathcal{H}_2/\mathcal{H}_{\infty}$ synthesis benchmarks on COMP$\textrm{l}_{\textrm{e}}$ib plants}\label{tb:H2Hinf_problems}
\newcommand{\cell}[1]{{\!\!\!}#1{\!\!\!\!}}
\newcommand{\cellbf}[1]{{\!\!\!}\textbf{#1}{\!\!\!}}
\begin{tabular}{|l|r|r|r|r|r|r|r|}\hline
\multicolumn{1}{|c|}{\!\!\!\!\! Prob.\!\!\!\!} & \multicolumn{3}{c|}{\!\!\!\!\!{Results ($\gamma=4$)}\!\!\!\!\!} &
\multicolumn{3}{|c|}{\!\!\!\!\!{Results ($\gamma=10$)}\!\!\!\!\!} \\ \hline
{\!\!}\texttt{Name}\!\!\!\!\! &\!\!\!$\mathcal{H}_2/\mathcal{H}_{\infty}$~~~~~~\!\!\!& \!\!\!\!\!\textrm{iter}\!\!\!\! & \!\!\!\!\texttt{time\![s]}\!\!\!\! &
\!$\mathcal{H}_2/\mathcal{H}_{\infty}$~~~~~\!&
\!\!\!\!\!\! \texttt{iter}\!\!\! \!\!\! & \!\!\!\!\texttt{time\![s]}\!\!\!\! \\ \hline
\cell{AC1} & \cellbf{0.0587/0.0993}  & \cell{ 2} & \cell{ 2.410}  & \cellbf{0.0587/0.0994} & \cell{   1} & \cell{ 2.000} \\ \hline
\cell{AC2} & \cellbf{0.1071/0.1730}  & \cell{ 1} & \cell{ 2.920}  & \cellbf{0.1071/0.1730} & \cell{   1} & \cell{ 2.720} \\ \hline
\cell{AC3} & \cellbf{-/-}            & \cell{-}  & \cell{-}       & \cellbf{4.5720/5.1337} & \cell{  57} & \cell{94.620} \\ \hline
\cell{AC6} & \cellbf{-/-}            & \cell{- } & \cell{-}       & \cellbf{3.9951/5.3789} & \cell{  28} & \cell{61.460} \\ \hline
\cell{AC7} & \cellbf{0.0438/0.0610}  & \cell{34} & \cell{50.080}  & \cellbf{0.0441/0.0611} & \cell{   3} & \cell{ 6.110} \\ \hline
\cell{AC11}& \cellbf{4.0914/3.9983}  & \cell{110}& \cell{150.340} & \cellbf{-/-}           & \cell{-}    & \cell{-} \\ \hline
\cell{AC12}& \cellbf{0.0924/0.3486}  & \cell{-}  & \cell{73.46}   & \cellbf{-/-}           & \cell{-}    & \cell{-} \\ \hline
\cell{AC17}& \cellbf{-/ - }          & \cell{-}  & \cell{-}       & \cellbf{4.2061/6.6126} & \cell{ 165} & \cell{100.130} \\ \hline
\cell{HE1} & \cellbf{0.0973/0.2046}  & \cell{1}  & \cell{34.860}  & \cellbf{0.0973/0.2075} & \cell{   1} & \cell{35.260} \\ \hline
\cell{HE2} & \cellbf{- / -}          & \cell{-}  & \cell{-}       & \cellbf{4.7326/9.8059} & \cell{ 135} & \cell{97.560} \\ \hline
\cell{REA1}& \cellbf{1.8217/1.4795}  & \cell{51} & \cell{23.140}  & \cellbf{1.8296/1.4495} & \cell{ 300} & \cell{172.700} \\ \hline
\cell{REA2}& \cellbf{3.5021/3.5122}  & \cell{72} & \cell{36.630}  & \cellbf{3.5024/3.4913} & \cell{ 141} & \cell{107.180} \\ \hline
\cell{DIS1}& \cellbf{- / -}          & \cell{-}  & \cell{-}       & \cellbf{4.2341/4.6736} & \cell{  44} & \cell{275.280} \\ \hline
\cell{DIS2}& \cellbf{1.5080/1.8410}  & \cell{45} & \cell{17.960}  & \cellbf{1.5080/1.8400} & \cell{  45} & \cell{20.280} \\ \hline
\cell{DIS3}& \cellbf{2.0580/1.7969}  & \cell{60} & \cell{68.530}  & \cellbf{2.0579/1.7727} & \cell{  66} & \cell{136.280} \\ \hline
\cell{DIS4}& \cellbf{1.6932/1.1899}  & \cell{72} & \cell{69.000}  & \cellbf{1.6932/1.1899} & \cell{  72} & \cell{68.120} \\ \hline
\cell{AGS} & \cellbf{- ~ / ~ -}        & \cell{-}  & \cell{-}       & \cellbf{7.0356/8.2053} & \cell{   9} & \cell{82.160} \\ \hline
\cell{PSM} & \cellbf{1.5157/0.9268}  & \cell{237}& \cell{241.210} & \cellbf{1.5158/0.9269} & \cell{ 264} & \cell{281.580} \\ \hline
\cell{EB2} & \cellbf{0.9023/0.8142}  & \cell{1}  & \cell{124.200} & \cellbf{0.9012/0.8142} & \cell{   1} & \cell{122.170} \\\hline
\cell{EB3} & \cellbf{0.9144/0.8143}  & \cell{1}  & \cell{123.470} & \cellbf{0.9137/0.8143} & \cell{   1} & \cell{126.810} \\\hline
\cell{NN2} & \cellbf{1.5652/2.4771}  & \cell{18} & \cell{20.540}  & \cellbf{1.5651/2.4811} & \cell{  24} & \cell{37.010} \\ \hline
\cell{NN4} & \cellbf{1.8778/2.0501}  & \cell{202}& \cell{154.49}  & \cellbf{1.8928/2.2496} & \cell{ 257} & \cell{139.900} \\ \hline
\cell{NN8} & \cellbf{2.3609/3.9999}  & \cell{21} & \cell{15.71}   & \cellbf{2.3383/4.5520} & \cell{  99} & \cell{68.700} \\ \hline
\cell{NN15}& \cellbf{0.0490/0.1366}  & \cell{24} & \cell{52.410}  & \cellbf{0.0488/0.1392} & \cell{  27} & \cell{49.940} \\ \hline
\cell{NN16}& \cellbf{0.3544/0.9569}  & \cell{108}& \cell{126.160} & \cellbf{0.3910/0.9573} & \cell{ 300} & \cell{405.340} \\ \hline
\end{tabular}
\end{scriptsize}
\end{center}
\end{table}
Here, $\mathcal{H}_2/\mathcal{H}_{\infty}$ are the $\mathcal{H}_2$ and $\mathcal{H}_{\infty}$ norms of the closed-loop systems for the static output feedback
controller, respectively.
With $\gamma = 10$, the computational results show that Algorithm \ref{alg:A1} satisfies the condition $\norm{P_{\infty}(s)}_{\infty} \leq \gamma = 10$ for all
the test problems. the problems AC11 and AC12 encounter a numerical problems that Algorithm \ref{alg:A1} can not solve. 
While, with $\gamma = 4$, there are $6$ problems reported infeasible, which are denoted by ``-''. The $\mathcal{H}_{\infty}$-constraint of three problems AC11
and NN8 is active with respect to $\gamma = 4$.

\section{Concluding remarks}
We have proposed a new iterative procedure to solve a class of nonconvex semidefinite programming problems. The key idea is to locally
approximate the nonconvex feasible set of the problem by an inner convex set. The convergence of the algorithm to a stationary point is investigated under
standard assumptions. We limit our applications to optimization problems with BMI constraints and provide a particular way to compute the inner psd-convex
approximation of a BMI constraint. Many applications in static output feedback controller design have been shown and two numerical examples have been
presented. Note that this method can be extended to solve more general nonconvex SDP problems where we can manage to find an inner psd-convex approximation of
the feasible set. This is also our future research direction.

\vskip 0.1cm
\noindent{\scriptsize
{\textbf{Acknowledgment. }
Research supported by Research Council KUL: CoE EF/05/006 Optimization in Engineering(OPTEC), IOF-SCORES4CHEM, GOA/10/009 (MaNet), GOA/10/11, several
PhD/postdoc and fellow grants; Flemish Government: FWO: PhD/postdoc grants, projects G.0452.04, G.0499.04, G.0211.05, G.0226.06, 
G.0321.06, G.0302.07, G.0320.08, G.0558.08, G.0557.08, G.0588.09,G.0377.09, research communities (ICCoS, ANMMM, MLDM); IWT: PhD Grants, Belgian
Federal Science Policy Office: IUAP P6/04; EU: ERNSI; FP7-HDMPC, FP7-EMBOCON, Contract Research: AMINAL. Other: Helmholtz-viCERP,
COMET-ACCM, ERC-HIGHWIND, ITN-SADCO.
}}

\end{document}